\documentclass[12pt,a4paper]{article}
\usepackage{graphicx}
\usepackage{amsmath}
\usepackage{amssymb}
\usepackage{amsthm}
\usepackage{enumitem}
\usepackage{bbm}
\usepackage[l2tabu,orthodox]{nag}
\usepackage[all,error]{onlyamsmath}
\usepackage[strict]{csquotes}
\usepackage{url}

\newtheorem{theorem}{Theorem}
\newtheorem{lemma}[theorem]{Lemma}
\newtheorem{corollary}[theorem]{Corollary}
\theoremstyle{definition}

\numberwithin{equation}{section}
\numberwithin{theorem}{section}

\newenvironment{OMabstract}{\noindent\textbf{Abstract.} }{\medskip}

\begin{document}

\author{D.A. Baranov, V.Z. Grines, O.V. Pochinka and E.E. Chilina} 
\title{On classification of periodic maps on the 2-torus}
\maketitle

\begin{center}

HSE University

\end{center}

\begin{OMabstract}
In this paper, following J.Nielsen, we introduce  a complete characteristic of orientation preserving periodic maps on the two-dimensional torus. All admissible complete characteristics were found and realized. In particular, each of classes of non-homotopic to the identity  orientation preserving  periodic homeomorphisms on the 2-torus   is realized by an algebraic automorphism. Moreover, it is shown that number of such classes is finite.  Due to  V.Z. Grines and A. Bezdenezhnykh, any gradient like orientation  preserving diffeomorphism of an orientable surface is represented as a superposition of the time-1 map of a gradient-like flow and some periodic homeomorphism. So, results of this  work  are directly  related to the complete topological classification of gradient-like diffeomorphisms on surfaces. 
\end{OMabstract}


\section{Introduction}

	  In \cite{Nielsen37},   J.Nielsen  found necessary and sufficient conditions of  topological conjugacy of periodic transformations of closed orientable surfaces. Describing of  all topological classes  for periodic maps is a difficult and boundless task. However,  this problem  was completely solved in the case of the two-dimensional sphere by Kerkyarto in \cite{Ker} and was partially solved in the case of the two-dimensional torus  by Brauer in  \cite{Br}.  In the present paper we describe all classes of topological conjugacy for periodic maps on the two-dimensional torus by means of  orientation preserving homeomorphisms. Moreover, the work contains  realization of  non-homotopy to the identity periodic maps on the two-dimensional torus by algebraic automorphisms. In addition, it is shown that number of such classes is finite. Notice that there are an infinite set of topological conjugacy classes for homotopic to the identity periodic maps of the two-dimensional torus. The realization of such classes is presented in \cite{KR}.

	Let $S$ be a closed orientable surface. Let us recall  that homeomorphisms $f,f':S\to S$ are called {\it topologically conjugate}  if there is a homeomorphism $h:S\to S$ such that $f'$ = $h\circ f\circ h^{-1}$. A homeomorphism $f$ is called periodic of period $n$  if $f^n = id$ and $f^m \neq id$ for each natural $m <n$.

	From the results of J. Nielsen \cite{Nielsen37}  the following statements are true for any orientation preserving periodic map $f$ on a closed orientable surface  $S$, whose period is $n$:

	\begin{enumerate}
		\item For each  $f$ we denote by  $\bar B_f\subset S$ a set of  {\it  periodic points of the homeomorphism $f$ with period strictly less than  $n$}. This set is either empty or consists of a finite number of orbits $\mathcal O_1,\dots, O_k$, $k\geq 1$. Denote by  $n_i$  period of a orbit $\mathcal O_i$, $i\in \{1, \dots, k\} $. Then $n_i$ is a divisor of $n$. Set $\lambda_i=\frac{n}{n_i}$, then for any orbit  $\mathcal O_i\subset\bar B_f$  there exists a unique number $\delta_i\in\{1,\dots,\lambda_i-1\}$ such that it is coprime to  $\lambda_i$   and there is  some neighborhood  $D_{\bar x_i}$ of a point $\bar x_i\in\mathcal O_i$ such that the restriction $f^{n_i}|_{D_{\bar x_i}}$ is topologically conjugate with the rotation by the angle  $\frac{2\pi \delta_i}{\lambda_i}$  of the complex plane around the origin:
		\begin{equation}\label{povorot}
		z \rightarrow e^{\frac{2\pi \delta_i}{\lambda_i}\bf{i}} z.
		\end{equation}
		\begin{figure}[h!]\center{\includegraphics[width=1\linewidth]{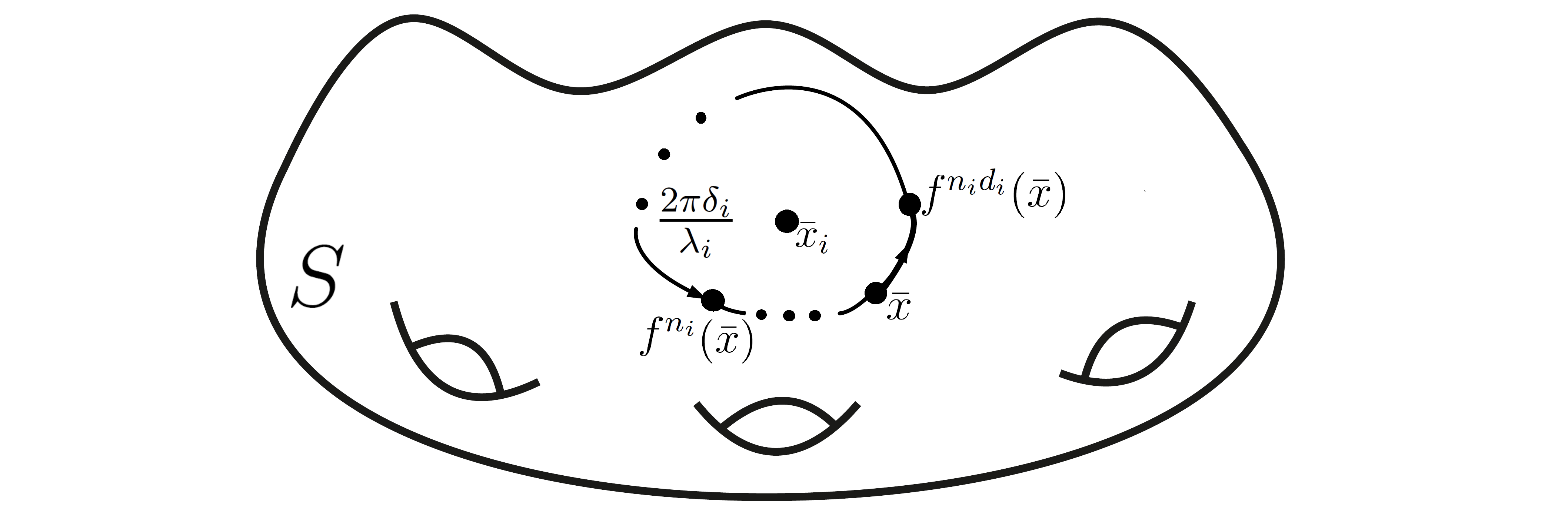}}\caption{  The action of the homeomorphism $f^{n_i}$ in some neighborhood of the point $\bar x_i$. }\label{tor}
\end{figure}

		The figure \ref{tor} illustrates the action of the homeomorphism $f^{n_i}$ in the neighborhood of the point $\bar x_i$.  This point has period $n_i$ with respect to the homeomorphism  $f$ and is  fixed point with respect to the homeomorphism $f^{n_i}$. The action of the homeomorphism $f^{n_i}$  is topologically conjugate to the rotation  by the angle $\frac{2\pi \delta_i}{\lambda_i}$ around the point in the counterclockwise direction.

		\item  For each  $\delta_i$ denote by  $d_i\in\{1,\dots,\lambda_i-1\}$ the  number such that $d_i\delta_i\equiv 1\pmod {\lambda_i}$.  Due to conjugacy with the map (\ref{povorot}) there exists  a  curve which is  homeomorpic to  the circle and  such that it is invariant under the homeomorphism $f^{n_i}$ and bounds a  disk containing a point $\bar x_i\in\bar B_f$. Then the number  $d_i$  has the following property: the open arc of the invariant curve  joining the  points  $\bar x$ and $ f^{n_id_i}(\bar x)$ (in the counterclockwise direction)  does not contain points of the orbit of point $\bar x$. A pair of numbers $(n_i,d_i)$   is called the {\it valency} of the orbit  $\mathcal O_i$.
		
		The figure \ref{tor} illustrates the closed curve. This is invariant with respect to the homeomorphism $f^{n_i}$ and contains each of points of the orbit of the point $\bar x$ under $f^{n_i}$. The open arc of the curve coonecting  points $\bar x$  and $ f^{n_id_i}(\bar x)$ (in the counterclockwise direction) does not contain points of this orbit.
		
		\item For each $f$ denote by $G=\{id,f,\dots,f^{n-1}\}$ the group which is  isomorphic to $\mathbb Z_n=\{0,\dots,n-1\}$. The orbit space $\Sigma=S/G$  is   called the {\it modular surface}. Also this is a closed surface, and the {\it natural projection} $\pi:S\to\Sigma$ is an $n$-fold covering map everywhere except the points of the set $\bar B_f$. Put  $x_i=\pi(\mathcal O_i)$.

	\end{enumerate}

	Let $p$ be the genus of the surface $S$ and $g$ be the genus of the modular surface $\Sigma$.

	For each periodic transformation $f$ of a closed orientable surface $S$ we define the collection of the numbers \begin{equation*}\kappa= (n, p, n_1,..., n_k, d_1,..., d_k), \end{equation*}
	 
	which we call the {\it complete characteristic} of the periodic transformation $f$.

		Denote by $gcd(a_1,\dots,a_m)$ the greatest common divisor of integers $a_1,\dots,a_m$.

	Next statements follow from above  mentioned the  works of Kerkjarto, Brouwer and Nielsen

	\newtheorem{statement}{Statement}
	
		\begin{statement}\label{exis}
		There is a periodic homeomorphism whose complete characteristic  is $\kappa= (n, p, n_1,..., n_k, d_1,..., d_k)$ and  genus of the modular surface is $g$ if and only if the following conditions are satisfied:
		
		\begin{enumerate}
		\item \begin{equation}\label{form2}
		\sum\limits_{i=1}^kd_in_i\equiv 0\pmod n 
	\end{equation}

\item	\begin{equation}\label{form}
		2p +  \sum\limits_{i=1}^k n_{i} - 2 = n(2g + k - 2).
	\end{equation}

		\item 	if $g=0$: \begin{equation}\label{sph}
	gcd(n_1d_1,\dots,n_kd_k,n)=1
	\end{equation}
		\end{enumerate}
		
	\end{statement}

		\begin{statement}\label{clNi}
		Two periodic maps $f$ and $f'$ on the orientable surface $S$ are topologically conjugate by means of an orientation preserving homeomorphism if and only the complete characteristic of $f$ coincides with  the complete characteristic of $f'$.
	\end{statement}
	
	If $\bar B_f=\emptyset$ the periodic transformation $f$ is completely described by the set of the numbers $(n, p)$. In this case the natural projection $\pi:S \to\Sigma$ is an $n$-fold covering map of the modular surface $\Sigma$ of genus $g$ by the surface $S$ of genus $p$. This implies the following statement.
	
	\begin{statement}\label{free} Two periodic transformations $f$, $f'$ of the surface $S$ such that  $\bar B_f=\emptyset$, $\bar B_{f'}=\emptyset$  are topologically conjugate by means of an orientation preserving homeomorphism if and only if $f$ and $f'$ have the same periods.
	\end{statement}

The main results of this work are the following theorems.

	\begin{theorem}\label{t2} There is an orientation preserving periodic homeomorphism  $f:\mathbb T^2\to\mathbb T^2$    such that $\bar B_f\neq\emptyset$ if and only if the complete characteristic of $f$ coincides exactly with one of seven complete characteristics:
			\begin{enumerate}
			\item $\kappa_1$: $n=2, \, p=1, \,n_1=n_2=n_3=n_4=1,\,d_1=d_2=d_3=d_4=1$;
			\item$\kappa_2$:  $n=3, \, p=1, \,n_1=n_2=n_3=1,\,d_1=d_2=d_3=1$;
			\item $\kappa_3$: $n=3,\, p=1 , \,n_1=n_2=n_3=1,\,d_1=d_2=d_3=2$;
			\item $\kappa_4$: $n=6, \, p=1,\,  n_1=3, n_2=2, n_3=1,\, d_1=d_2=d_3=1$;
			\item $\kappa_5$: $n=6,\, p=1 ,  \, n_1=3, n_2=2, n_3=1,\,d_1=1,\,d_2=2,\,d_3=5$;
			\item $\kappa_6$: $n=4,\, p=1 , \,n_1=2,\,n_2=n_3=1,\,d_1=d_2=d_3=1$;
			\item $\kappa_7$: $n=4,\, p=1,\,n_1=2,\,n_2=n_3=1,\,d_1=1,\,d_2=d_3=3$.
		\end{enumerate}  
	\end{theorem}

		\begin{theorem}\label{t1} Let $f:\mathbb T^2\to\mathbb T^2$ be an  orientation preserving periodic homeomorphism  of period $n\in\mathbb N$. Then the following conditions are equivalent:
		\begin{enumerate}
			\item $f$ is homotopic to the identity;
			\item $\bar B_f=\emptyset$;
			\item $g=1$;
			\item $f$ is topologically conjugate to the shift on the torus $\Psi_{n}\left(e^{i2x\pi},e^{i2y\pi}\right)= \left(e^{i2\pi\left(x+\frac{1}{n}\right)},e^{i2y\pi}\right).$
		\end{enumerate}  
	\end{theorem}

	Let $A=\begin{pmatrix} 
    a & b \\
    c & d 
\end{pmatrix}\in Gl(2,\mathbb{Z})$, that is, $A$ is a second-order integer square matrix  with $\det A=\pm1$. Then, $A$ induces the map  $f_A : \mathbb{T}^2 \rightarrow \mathbb{T}^2$ given by the formula

\begin{equation*} f_A:
 \begin{cases}
   \overline{x}=ax+by \pmod 1 
   \\
  \overline{y}=cx+dy \pmod 1
 \end{cases},
 \end{equation*}
which  is  an {\it algebraic automorphism of the two-dimensional torus}.

	In  section \ref{Kat} we give a classification of algebraic automorphisms of the two-dimensional torus and use them to realize the topological classes of periodic transformations of the two-dimensional torus.

	\begin{theorem}\label{t3} Any non-homotopic to  the identity orientation preserving periodic homeomorphism $f:\mathbb T^2\to\mathbb T^2$ is conjugate by means of an orientation preserving homeomorphism with exactly one map  $f_{A_j}$ induced by the matrix $A_j$ $(j=\overline{1,7} )$:
\begin{center}

 $A_1=\begin{pmatrix}
			-1 & 0\\
			0 & -1\\
		\end{pmatrix}$;
		 $A_2=\begin{pmatrix}
			-1 & -1\\
			1 & 0\\
		\end{pmatrix}$;	
 $A_3=\begin{pmatrix}
			0 & 1\\
			-1 & -1\\
		\end{pmatrix}$;
		$A_{4}=\begin{pmatrix}
			0 & -1\\
			1 & 1\\
		\end{pmatrix}$;
$A_5=\begin{pmatrix}
			1 & 1\\
			-1 & 0\\
		\end{pmatrix}$;
 $A_6=\begin{pmatrix}
			0 & -1\\
			1 & 0\\
		\end{pmatrix}$;
 $A_7=\begin{pmatrix}
			0 & 1\\
			-1 & 0\\
		\end{pmatrix}$.
\end{center}

The complete characteristic of $f_{A_j}$ coincides with a complete characteristic $\kappa_j$.
	\end{theorem}

\section{Auxiliary inequalities for finding  complete characteristics of periodic transformations}

	In this section we prove some useful consequences  from the formula (\ref{form})   $2p +  \sum\limits_{i=1}^k n_{i} - 2 = n(2g + k - 2)$.
	
	\begin{lemma}\label{co1}For any orientation preserving periodic homeomorphism $f: S \rightarrow\ S$  such that $\bar B_f=\emptyset$  the following equality holds:
		\begin{equation}\label{form3}
			p=n(g-1)+1
		\end{equation}
	\end{lemma}

\begin{proof}
Due to  $\bar B_f=\emptyset$, the number $k=0$. Using the formula (\ref{form}) $2p + \sum\limits_{i=1}^k{n_{i}}-2$ = $n(2g+k-2)$, we get $2p-2=n(2g-2)$. Therefore,  $p=n(g-1)+1$.
	
\end{proof}

\begin{lemma}\label{co12} For any orientation-preserving periodic homeomorphism $f: S \rightarrow\ S$  such that $\bar B_f\neq\emptyset$  the following equality holds:
		\begin{equation}\label{form4}
			p>n(g-1)+1
		\end{equation}
	\end{lemma}
	\begin{proof} By definition $n_i$ is a divisor of $n$ not exceeding $n$ for all $i=\overline{1,k}$, therefore, $n_i\leqslant\frac{n}{2}$ and the following relation holds:  \begin{equation}\label{form5}
			0<k\leqslant\sum\limits_{i=1}^k{n_{i}}\leqslant \frac{nk}{2}<nk.
		\end{equation} Let consider all possible cases:
		\begin{enumerate}
			\item $p=0.$ Then, equality (\ref{form})  is equivalent to equality $\sum\limits_{i=1}^k{n_{i}}-2=n(2g+k-2)$. Transforming this, we get: $\sum\limits_{i=1}^k{n_{i}}-2=n(2g-2)+nk$.  It follows from  (\ref{form5})  that $nk-2>n(2g-2)+nk$. Hence, $n(2g-2)<-2$ and $n(1-g)>1$. Then, $1-g>0$. Considering that $g\in\mathbb{N}$, we get $g=0$.
			\item $p=1$. Then, equality (\ref{form}) is equivalent to equality $\sum\limits_{i=1}^k{n_{i}}=n(2g+k-2)$. Using (\ref{form5}), we obtain $n(2g+k-2)<nk$. Therefore, $g-1<0$ and $g<1$. Considering  $g\in\mathbb{N}$, we get $g=0$.
			\item $p>1$. Using (\ref{form}) and  (\ref{form5}),  we obtain that $n(2g+k-2)<2p+nk-2$. Whence, $n(g-1)<p-1$ and  $p>n(g-1)+1$.
		\end{enumerate}
		Thus, in all cases the following inequality holds: $p>n(g-1)+1$.

\end{proof}
	
	\begin{lemma}\label{co2} For any orientation-preserving periodic homeomorphism $f:\mathbb T^2 \rightarrow\mathbb T^2$  such that $\bar B_f\neq\emptyset$ the following equality holds:
		\begin{equation}\label{form6}
			2<\frac{2n}{n-1}\leqslant k\leqslant 4
		\end{equation}
	\end{lemma}
\begin{proof} In the case of the two-dimensional torus $p=1$. Then, it follows from Lemma \ref{co12} that  the genus of the modular surface is equal to $0$. Therefore, equality (\ref{form}) is equivalent to the equality
		\begin{equation}\label{form7}
			nk = \sum\limits_{i=1}^k{n_{i}} + 2n
		\end{equation}	
		Using  inequality (\ref{form5}), we see that $2n+k \leqslant\ nk \leqslant\ 2n + \frac{nk}{2}$. From the left  side of this inequality we obtain that $k \geqslant \frac{2n}{n-1}>\frac{2n-2}{n-1}=2$.  From the right side we obtain $\frac{nk}{2} \leqslant\ 2n$. Therefore, $n(4-k)\geqslant 0$. Considering that $k, n\in\mathbb{N}$, we have $k\leqslant 4$.
	\end{proof}

	\section{Complete characteristics of periodic transformations of the  two-dimensional torus}

It follows from  statement \ref{free} that if the set $\bar B_f=\emptyset$  for the homeomorphism $f$,  then the complete characteristic of $f$ is determined by the genus $p$  of the surface and the period  $n$ of the transformation. Thus, complete characteristics of periodic homeomorphisms on the two-dimensional torus corresponding to the empty set $\bar B_f$ has the form $(n,1)$, where $n$ is the  period of $f$.
	
	In this section we find all  complete characteristics of an orientation preserving periodic homeomorphisms of the two-dimensional torus corresponding to the nonempty set $\bar B_f$ and  prove that there are only 7 such complete characteristics. Using transformations of a modular surface, we realize all such classes  by the periodic homeomorphism.

	{\bf Proof of Theorem \ref{t2}.}  {\bf Let us prove the necessity.} Let $f$ be an orientation preserving periodic homeomorphism of the two-dimensional torus such that $\bar B_f\neq\emptyset$.
	
		Let us find all admissible complete characteristics of the transformation $f$.  Due to Lemma \ref{co2}, we have that the number $k$ of orbits   of the set  $\bar B_f$ should only be equal to  $3$ or $4$. Consider these cases separately.

				1 case: $k=4$. Substituting $4$ for $k$ in (\ref{form7}), we get $2n = \sum\limits_{i=1}^4{n_{i}}$. Due to $n_{i} \leq \frac{n}{2}$ the equality (\ref{form7}) holds only if $n_{i}=\frac{n}{2}$ for all $i=1,\dots,4$.  Substituting $\frac{n}{2}$  for $n_i$ in (\ref{sph}), we get  $gcd(\frac{n}{2},\frac{n}{2},\frac{n}{2},\frac{n}{2},n)=1$. Therefore, $n=2$. Since  $n_i<n$, we have $n_i=1$ for each  $i=\overline{1,4}$. Hence, $\lambda_{i}=\frac{2}{1}=2$. Since $d_{i}<\lambda_{i}$, we have $d_{i}=1$ for each $i=\overline{1,4}$. We obtain the complete characteristic of $(2,1,4,1,1,1,1,1,1,1,1)$, which we denote by $\kappa_1$.

				2 case: $k=3$.  Substituting $3$ for $k$ in (\ref{form7}), we get $n = \sum\limits_{i=1}^3{n_{i}}$.
		From this it is obviously that at least one $n_i$ is not less than $\frac{n}{3}$. For definiteness put $n_{1}\geq\frac{n}{3}$. Consider 2 admissible cases: a) $n_{1}=\frac{n}{3}$ and b) $n_{1}=\frac{n}{2}$.
		
			$a)$ $n_{1}=\frac{n}{3}$. Then $n_{2}+n_{3}=\frac{2n}{3}$. Therefore, at least one of $n_2,\,n_3$ is not less than $\frac{n}{3}$. For definiteness put $n_{2}\geq\frac{n}{3}$. Then,  $a_1)$ $n_{2}=\frac{n}{3}$ and $n_{3}=\frac{n}{3}$, or $a_2)$ $n_{2}=\frac{n}{2}$ and $n_{3}=\frac{n}{6}$. 
		
		$a_1)$ Substituting $\frac{n}{3}$  for $n_i$ in (\ref{sph}), we get $gcd(\frac{nd_1}{3},\frac{nd_2}{3}, \frac{nd_3}{3}, n)=1$.  Hence, $n=3$. Then, substituting the data $n_{1}=n_{2}=n_{3}=1$ for  $n_i$ in (\ref{form2}), we  get $ \sum\limits_{i=1}^3d_i\equiv 0\pmod n$. Due to $d_i\in\{1,2\}$ there are 2 admissible number collections:  $d_{1}=d_2=d_3=1$ (the complete characteristic $(3,1,3,1,1,1,1,1,1)$, which we denote by  $\kappa_2$), or $d_{1}=d_2=d_3=2$ (the complete characteristic $(3,1,3,1,1,1,2,2,2)$, which we denote by $\kappa_3$).

		$a_2)$ Taking into account  (\ref{sph}), we get $n=6$. Using (\ref{form2}) and the fact that $d_i\in\{1,\dots,5\}$, we obtain 2 admissible number collections:  $d_{1}=d_2=d_3=1$ (the complete characteristic $(6,1,3,3,2,1,1,1,1)$, which we denote by $\kappa_4$) and $d_1=1,\,d_2=2,\,d_3=5$ (the complete characteristic $(3,1,3,1,1,1,1,2,5)$, which we denote by $\kappa_5$).

		$b)$ $n_{1}=\frac{n}{2}$. Then, $n_{2}+n_{3}=\frac{n}{2}$, whence $\frac{n}{2}>n_{2}\geq\frac{n}{4}$. Therefore, $b_1)$ $n_{2}=\frac{n}{3}$ and $n_{3}=\frac{n}{6}$, or $b_2)$ $n_{2}=\frac{n}{4}$ and $n_{3}=\frac{n}{4}$. 
		
		$b_1)$ The complete characteristics for this number collection is found in $(a_2)$. 
		
		$b_2)$ For this number collections, taking into account  (\ref{sph}), we have $n=4$. Using (\ref{form2}) and the fact that $d_i\in\{1,\dots,5\}$, we obtain 2 admissible number collections: $d_1=d_2=d_3=1$ (the complete characteristic $(4,1,3,2,1,1,1,1,1)$, which we denote by $\kappa_6$) and $d_1=1,\,d_{2}=d_{3}=3$ (the complete characteristic $(4,1,3,2,1,1,1,3,3)$, which we denote by $\kappa_7$).

			{\bf Let us prove the sufficiency.} Let's show that for  each  complete characteristic  $\kappa_i$ ($i=\overline{1,7}$) there is an  orientation preserving periodic homeomorphism of the two-dimensional torus. Using the formula \ref{form}, we get that for such complete characteristics the genus of the modular surface is equal to $0$. Therefore, a modular surface is the sphere.

			Let us construct the homeomorphism $f_1$ satisfying the found complete characteristic $\kappa_1$. We mark on the sphere (modular surface) 4 points: $x_1, x_2, x_3, x_4$, each of which is the projection of  orbits with period less than the period of   the homeomorphism  $f_3$. Construct an arc with the beginning at the point $x_1$ and the end at the point $x_4$ on the sphere such that it contains points $x_2$ and $x_3$ and the point $x_2$ locates between the points $x_1$ and $x_3$ (see Fig. \ref{2per}$a$). Cutting the sphere along the constructed arc, we obtain a disk with the boundary $x_1x_2x_3x_4x_3x_2x_1$   (see Fig.\ref{2per}$b$). Gluing two such disks along the boundary $x_3x_4x_3$, we get a square (see Fig. \ref{2per}$c$). Having identified the sides of the square as shown in the figure \ref{2per}$c$), we obtain the two-dimensional torus (see Fig. \ref{2per}$d$).

	Define the homeomorphism $f_1$ by the following rule. Rotating the disk 1 in the figure \ref{2per}$c$) by the angle $\pi$ in the counterclockwise direction, we map it  into the disk 2. Similarly, we  map the disk 2 into the disk 1 by the angle $\pi$ in the counterclockwise direction. The described map is an orientation preserving periodic homeomorphism of the two-dimensional torus satisfying the complete characteristic  $\kappa_1$.

		 \begin{figure}[h!]\center{\includegraphics[width=1\linewidth]{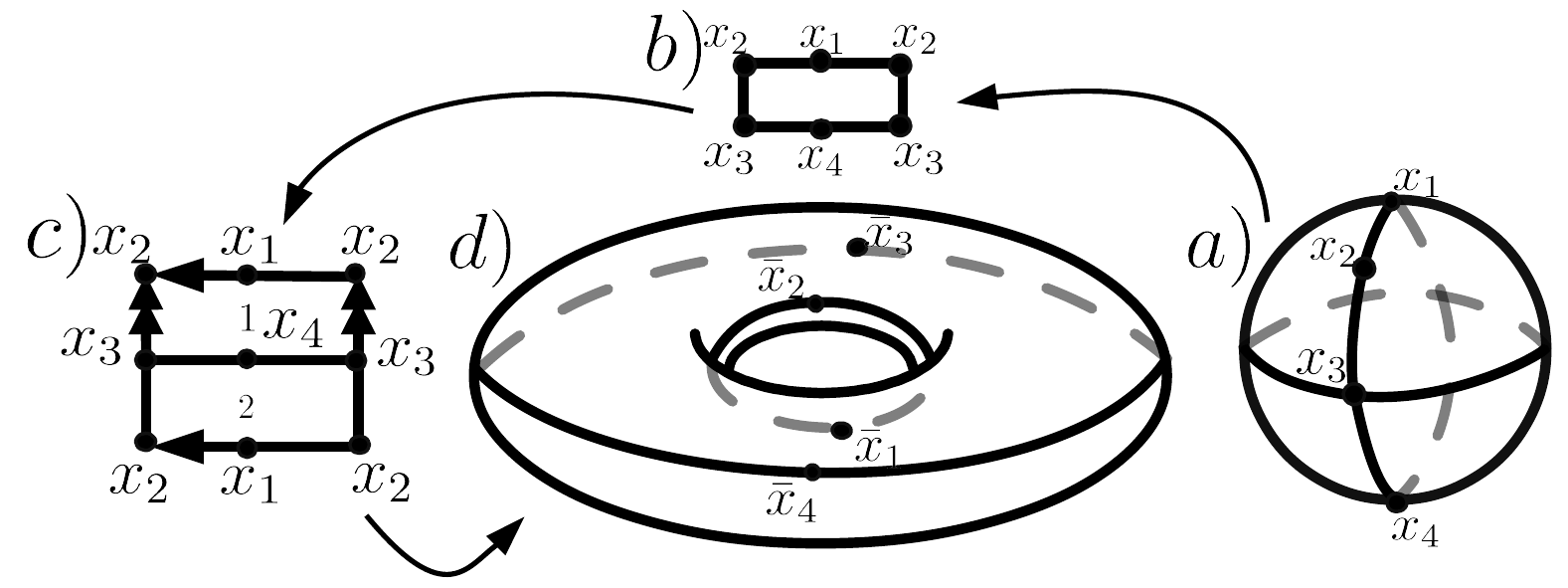}}\caption{The construction of the homeomorphism with the complete characteristic $\kappa_1$.}\label{2per}\end{figure}
	
		Let us construct homeomorphisms $f_2$ and $f_3$ satisfying  found complete characteristics $\kappa_2$ and  $\kappa_3$ respectively .We mark 3 points on the sphere: $x_1, x_2, x_3$, each of which is the projection of  orbits with period less than the period of  homeomorphisms  $f_2$ and $f_3$. Construct an arc with the beginning at the point $x_1$ and the end at the point $x_3$ on the sphere such that it contains the point $x_2$  (see Fig. \ref{3per}$a$). Cutting the sphere along the constructed arc, we obtain a disk with the boundary $x_1x_2x_3x_2x_1$   (see Fig.  \ref{3per}$b$). Gluing in pairs  three such disks along the boundary $x_1x_2$, we obtain a hexagon  (see Fig.  \ref{3per}$c$). Having identified the sides of the hexagon as shown in the figure \ref{3per}c), we obtain the two-dimensional torus (see Fig. \ref{3per}$d$). 
		 	
		 	Define the homeomorphism $f_2$ by the following rule.   Rotating the disk 1  in the figure \ref{3per}$c$)   by the angle $\frac{2\pi}{3}$in the counterclockwise direction, we map it  into the disk 2. Similarly, we map the  disk 2 into the disk 3, and the disk 3 into the disk 1. Figure \ref{3per}$e_1$) illustrates the action of the  map in the neighborhood of a fixed point. The described map is an orientation preserving periodic homeomorphism of the two-dimensional torus satisfying the complete characteristic $\kappa_2$. 
		 	
		 	 	Define the homeomorphism $f_3$ by the following rule.   Rotating the disk 1  in the figure \ref{3per}$c$)  by the angle $\frac{4\pi}{3}$ in the counterclockwise direction, we map it  into the disk 3. Similarly, we map the  disk 3 into the disk 2, and the disk 2 into the disk 1. Figure \ref{3per}$e_1$) illustrates the action of the  map in the neighborhood of a fixed point. The described map is an orientation preserving periodic homeomorphism of the two-dimensional torus satisfying the complete characteristic $\kappa_3$.

		 \begin{figure}[h!]\center{\includegraphics[width=1\linewidth]{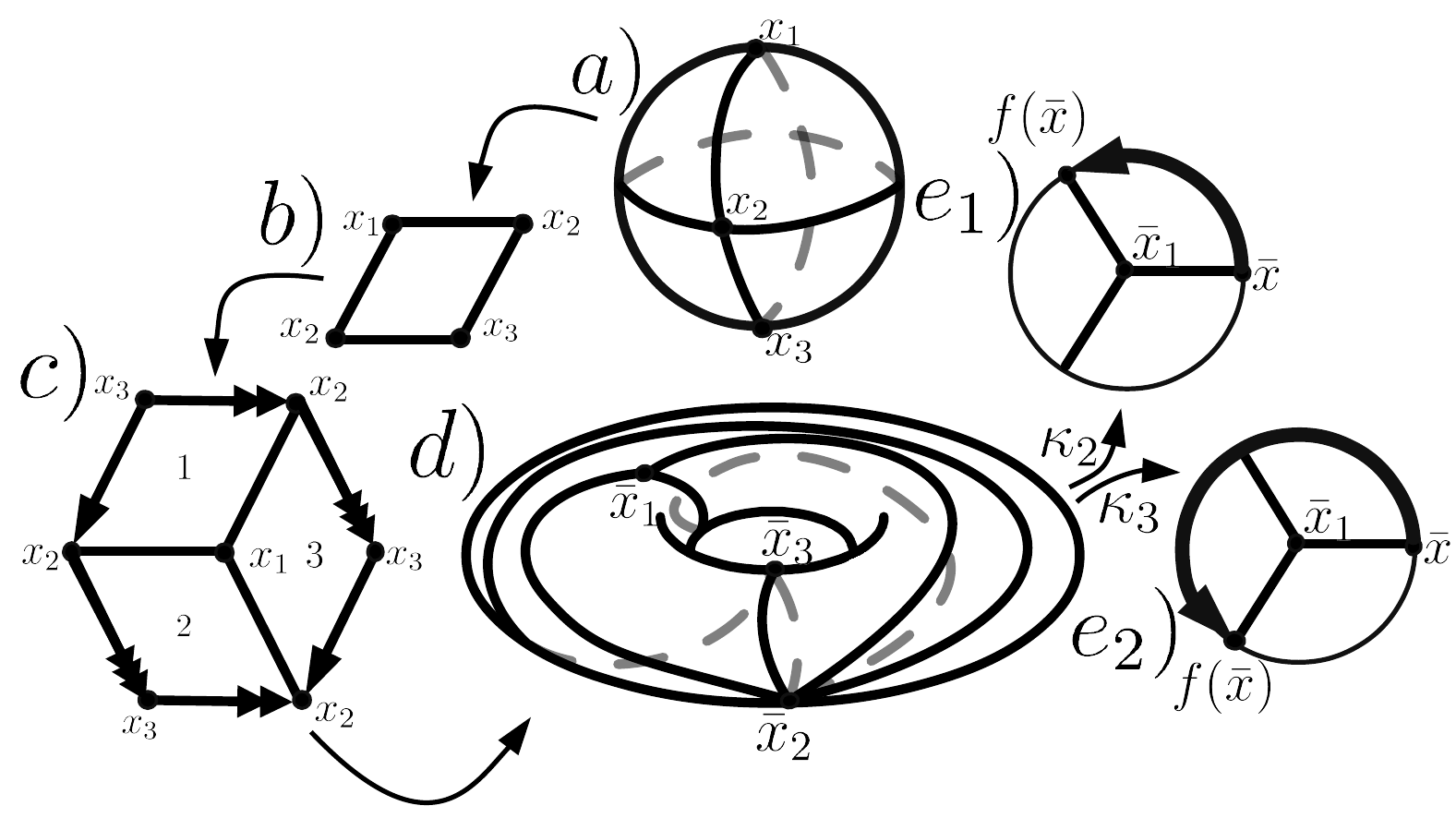}}\caption{The construction of homeomorphisms with complete characteristics $\kappa_2$ and $\kappa_3$. }\label{3per}\end{figure}

		Let us construct homeomorphisms $f_4$ and $f_5$  satisfying the found complete characteristics $\kappa_4$ and  $\kappa_5$ respectively. We mark 3 points on the sphere: $x_1, x_2, x_3$, each of which is the projection of orbits with period less than the period of   homeomorphisms  $f_4$ and $f_5$. Construct an arc with the beginning at the point $x_2$ and the end at the point $x_3$ on the sphere such  that it contains the point $x_2$  (see Fig. \ref{6per}$a$). Cutting the sphere along the constructed arc, we obtain a disk with the boundary $x_3x_1x_2x_1x_3$   (see Fig.  \ref{6per}$b$). Gluing in pairs six such disks  along the boundary $x_1x_3$, we get a polygon (see Fig. \ref{6per}$c$). Having identified the sides of  the polygon as shown in the figure \ref{6per}$c$), we obtain the two-dimensional torus.

		 	Define the homeomorphism $f_4$ by the following rule.   Rotating the disk 1 in the figure \ref{6per}$c$) by the angle $\frac{\pi}{3}$  in the counterclockwise direction, we map the disk 1 into the disk 2. Similarly, we  map the disk 2 into the disk 3, the disk 3 into disk 4, the disk 4 into the disk 5, the disk 5 into the disk 6, and the disk 6 into the disk 1. Figure \ref{6per}$d_1$) illustrates the location of  orbits with period less than the period of the homeomorphism $f_4$. The described map is an orientation preserving periodic homeomorphism of the two-dimensional torus satisfying the complete characteristic  $\kappa_4$.

		 	Define the homeomorphism $f_5$ by the following rule.   Rotating the disk 1 in the figure \ref{6per}$c$) by the angle $\frac{5\pi}{3}$  in the counterclockwise direction, we map the disk 1 into the disk 6. Similarly, we  map the disk 6 into the disk 5, the disk 5 into disk 4, the disk 4 into the disk 3, the disk 3 into the disk 2, and the disk 2 into the disk 1. Figure \ref{6per}$d_2$) illustrates the location of  orbits with period less than the period of the homeomorphism $f_5$. The described map is an orientation preserving periodic homeomorphism of the two-dimensional torus satisfying the complete characteristic  $\kappa_5$.

		 	 \begin{figure}[h!]\center{\includegraphics[width=1\linewidth]{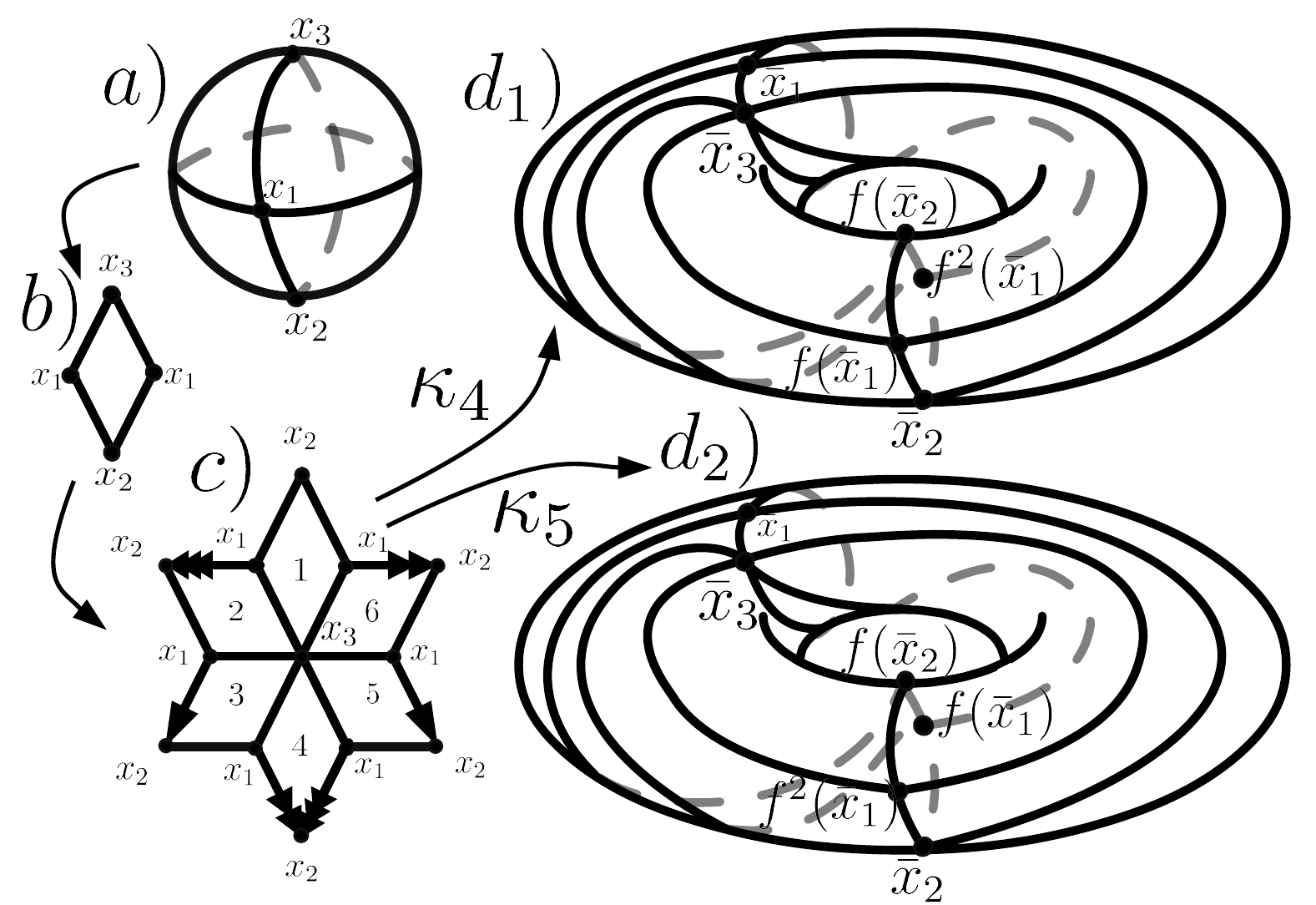}}\caption{The construction of homeomorphisms with complete characteristics $\kappa_4$ and $\kappa_5$. }\label{6per}\end{figure}

		Let us construct homeomorphisms $f_6$ and $f_7$ satisfying  found complete characteristics $\kappa_6$ and  $\kappa_7$ respectively. We mark 3 points on the sphere: $x_1, x_2, x_3$, each of which is the projection of  orbits with period less than the period of   homeomorphisms  $f_6$ and $f_7$. Construct an arc with the beginning at the point $x_2$ and the end at the point $x_3$ on the sphere such that it contains the point $x_3$ (see Fig. \ref{4per}$a$). Cutting the sphere along the constructed arc, we obtain a disk with the boundary $x_3x_1x_2x_1x_3$   (see Fig. \ref{4per}$b$). Gluing in pairs four such disks  along the boundary $x_1x_3$, we get a square (see Fig. \ref{4per}$c$).  Having identified the sides of  the square as shown in the figure \ref{4per}$c$), we obtain the two-dimensional torus (see Fig. \ref{4per}$d$). 
		 	
		 		Define the homeomorphism $f_6$ by the following rule.   Rotating the disk 1 in the figure \ref{4per}$c$)  by the angle $\frac{\pi}{4}$  in the counterclockwise direction, we map the disk 1 into the disk 2. Similarly, we map the disk 2 into the disk 3, the disk 3 into the disk 4, and the disk 4 into the disk 1. Figure \ref{4per}$e_1$) illustrates the action of the  map in the neighborhood of a fixed point.  The described map is an orientation preserving periodic homeomorphism of the two-dimensional torus satisfying the complete characteristic $\kappa_6$. 
		 		
		 			Define the homeomorphism $f_7$ by the following rule.   Rotating the disk 1  in the figure \ref{4per}$c$) by the angle $\frac{3\pi}{4}$   in the counterclockwise direction, we map the disk 1 into the disk 4. Similarly, we map the disk 4 into the disk 3, the disk 3 into the disk 2, and the disk 2 into the disk 1. Figure \ref{4per}$e_2$) illustrates the action of the  map in the neighborhood of a fixed point.  The described map is an orientation preserving periodic homeomorphism of the two-dimensional torus satisfying the complete characteristic $\kappa_7$.

		 		\begin{figure}[h!]\center{\includegraphics[width=1\linewidth]{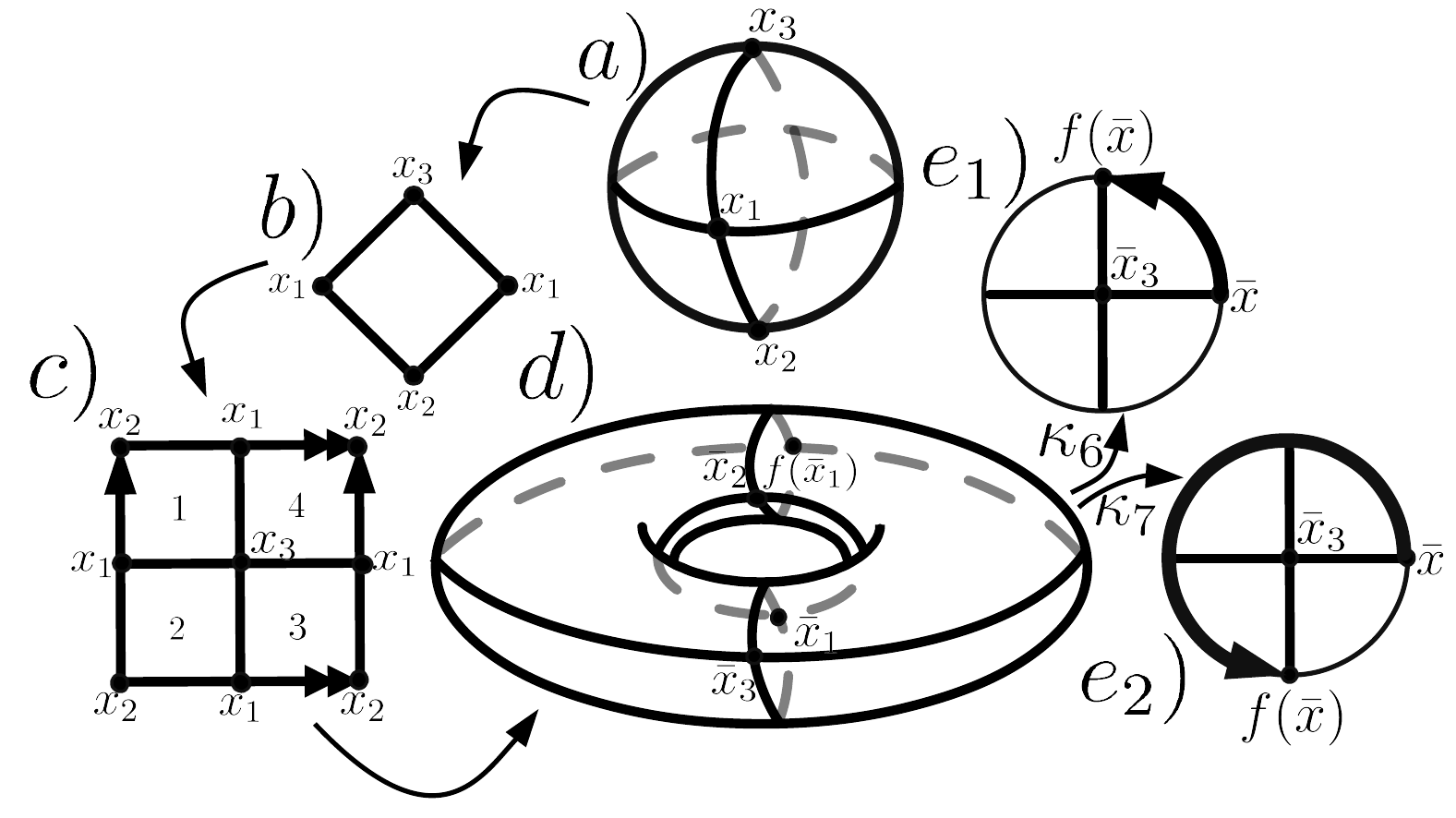}}\caption{The construction of homeomorphisms with complete characteristics  $\kappa_6$ and $\kappa_7$. }\label{4per}\end{figure}

\section{Classification of  homotopic to the identity periodic homeomorphisms of the two-dimensional torus }

	In this section we prove Theorem \ref{t1}.

Let us recall the concept of the index of a fixed point of a map and the Lefschetz-Hopf formula for  homotopic  to the identity map with a finite number of fixed points.

	Let the vector field $\xi(x,y)$ be defined and continuous at each of points of the plane $\mathbb{R}^2$ except maybe some points. Denote by $M_0(x_{0},y_0)$ the singular point of the vector field $\xi(x,y)$.  Choose   $r>0$ such that the disk $d_r=\{(x,y)\in\mathbb R^2: \sqrt{(x-x_{0})^2+(y-y_0)^2}\leqslant r\}$  does not contain singular points of the vector field different from the point  $M_0(x_{0},y_0)$. Then, on the set of points $d_r$ the vector field  $\xi(x,y)$ is determined by the transformation $f$ and is given by the formula $\xi(x,y)=(x,y)-f(x,y)$. The number of rotations of the vector field $\xi|_{\partial d_r}$ in the counterclockwise direction in the neighborhood of the point $M_0$  is called  the {\it  index of the singular point} $M_{0}$ and is denoted by $I(M_{0})$. For a continuous map  defined on a surface the fixed point index is calculated as the index of the singular point of the local representation of the map  in an isolated neighborhood of this point.

	Consider $\lambda$ and $\delta$ such that $\lambda\in\mathbb{N}$, $\delta\in\{1,\dots,\lambda-1\}$ and $(\delta,\lambda)=1$. Then, the index of the fixed point of the plane rotation by  the angle $\frac{2\pi \delta}{\lambda}$  is equal to $\delta$  if $\delta\leqslant\frac{\lambda}{2}$ and is equal to $-\delta$ if $\delta>\frac{\lambda}{2}$. For the surface $S$ of the genus $p$  homotopic to the identity map $f:S\to S$ with a finite number of fixed points $x_{1}, ..., x_{k}$ we have {\it Lefschetz-Hopf formula} \cite{LX}:
	\begin{equation}\label{form8}
		\sum\limits_{i=1}^k I(x_i)=2-2p
	\end{equation}

	{\bf Proof of Theorem \ref{t1}.}  Let $f:\mathbb T^2\to\mathbb T^2$ be an  orientation preserving periodic homeomorphism  of period $n\in\mathbb N$. 
	
	Let us prove that the following implications are true: $1\rightarrow 2,\,2\rightarrow 3,\,3\rightarrow 4,\,4\rightarrow 1$.
		
		$1\rightarrow 2.$ Assume the converse. Then, $\bar B_f \neq \emptyset$. In this case taking into acount  Theorem \ref{t2}, we have 7 admissible complete characteristics $\kappa_j$ $(j=\overline{1,7})$  for the homeomorphism $f$. In each of these cases let $n_*$ be the largest period of orbits of $\bar B_f $. Consider the homeomorphism $\phi=f^{n_*}$. On the one hand, the homeomorphism $\phi$ is isotopic to the identity because $\phi$ is the degree of isotopic to the identity map. On the other hand, summing  indices of  fixed points, we get a nonzero value, which contradicts the Lefschetz-Hopf formula (\ref{form8}). This contradiction proves the implication.
		
		$2\rightarrow 3.$ If $\bar B_f=\emptyset$, then it follows from  Lemma \ref{co1}   that    $p-1=n(g-1)$. Substituting $1$ for $p$ in this equality, we obtain  $n(g-1)=0$. Therefore, $g=1$.
		
		$3\rightarrow 4.$  Substituting $1$ for $g$ and $1$ for $p$ in (\ref{form}), we get   $\sum\limits_{i=1}^k{n_{i}}=nk$.  Using \ref{form5}, we have    $\sum\limits_{i=1}^k{n_{i}}\leqslant\frac{nk}{2}$. Therefore, $k=0$ and  $\bar B_f=\emptyset$. It follows from statement  \ref{free} that the periodic homeomorphism $f$ of period $n$ is topologically conjugate to the diffeomorphism $\Psi_{n}\left(e^{i2x\pi},e^{i2y\pi}\right)= \left(e^{i2\pi\left(x+\frac{1}{n}\right)},e^{i2y\pi}\right)$, which has  period $n$. 
		
		$4\rightarrow 1.$ By construction, the diffeomorphism $\Psi_{n}\left(e^{i2x\pi},e^{i2y\pi}\right)= \left(e^{i2\pi\left(x+\frac{1}{n}\right)},e^{i2y\pi}\right)$ acts identically in the fundamental group and, therefore, it is isotopic to the identity map  (see, for example, \cite{Ro}).

	The following statement is a corollary to  Theorem \ref{t2} and Theorem \ref{t1}.
	
	\begin{corollary}\label{ng} Periodic homeomorphisms of the two-dimensional torus satisfying complete characteristics $\kappa_j$ $(j=\overline{1,7})$ and only they are non-homotopic to the identity preserving orientation periodic homeomorphisms of the two-dimensional torus. \end{corollary}

	In each topological conjugacy class of preserving orientation homotopic to the identity periodic homeomorphisms of the two-dimensional torus  there is the representative  $f_{\alpha}:\mathbb{T}^2\to \mathbb{T}^2$, which is a translation of the two-dimensional torus by the vector $\alpha=(\alpha_1,\alpha_2)$  defined by the formula $f_{\alpha}(\bar x, \bar y)=((x + \alpha_1) \pmod 1; (y + \alpha_2) \pmod 1)$, where $\alpha_1,\alpha_2\in\mathbb{Q}$.  Necessary and sufficient conditions for the conjugacy of two translations on the 2-torus are found in \cite{KR}.

	\section{Classification of periodic algebraic automorphisms of the two-dimensional torus}\label{Kat}

	Let $A\in Gl(2,\mathbb{Z})$ and  $f_A : \mathbb{T}^2 \rightarrow \mathbb{T}^2$   is  an  algebraic automorphism of the two-dimensional torus induces by $A$.

 If eigenvalues of $A\in Gl(2,\mathbb{Z})$ are not equal in modulus to unity, then the algebraic automorphism of the two-dimensional torus induced by the matrix $A$ is called a {\it hyperbolic} algebraic automorphism of the two-dimensional torus. Otherwise, the automorphism $f_A$ is called a  {\it non-hyperbolic} algebraic automorphism of the two-dimensional torus.

Two algebraic automorphisms of the 2-torus $f$ and $g$ are called {\it conjugate}, if there exists an automorphism $h$ such that $gh=hf$.  

The set $K_f=\{hfh^{-1} | h \;  {\text is  \;  an \; algebraic \;  automorphism \;  of \;  the  \;   2-torus}\} $ is called the  {\it  conjugacy class} of the automorphism $f$.

We denote by $\mathbb{Z}^{2\times 2} $  set of integer matrices of order $2$.   The matrix $B\in\mathbb{Z}^{2\times 2}$   is called a {\it similar over $\mathbb{Z}$} to the matrix $A\in\mathbb{Z}^{2\times 2} $  , if there exists   matrix $S\in Gl(2,\mathbb{Z})$ such that $B=S^{-1}AS$. 
If $A\in\mathbb{Z}^{2\times 2}$, then the set $K_A=\{S^{-1}AS\ |\ S\in Gl(2,\mathbb{Z})\} $ is called the {\it similarity class of the matrix $A$}.

Thus, the problem of finding the conjugacy classes of non-hyperbolic algebraic automorphisms of the two-dimensional torus is reduced to the problem of finding the similarity classes of second-order integer unimodular matrices, whose eigenvalues are equal in modulus to unity. The problem of finding similarity classes for second-order integer unimodular matrices whose eigenvalues are roots of unity was solved in \cite{Batterson} in the form of the following statement:

\begin{statement}\label{Batt}

Let $A\in Gl(2,\mathbb{Z})$ and suppose that both eigenvalues of $A$ are roots of unity. Then $A$ is similar over  $\mathbb{Z}$ to exactly one of the following matrices:

\begin{equation*}
\begin{pmatrix} 
   1 & m \\
   0 & 1 
\end{pmatrix} ,
\begin{pmatrix} 
   -1 & m \\
   0 & -1 
\end{pmatrix},
\begin{pmatrix} 
   1  & 0 \\
   0 & -1 
\end{pmatrix}
,
\begin{pmatrix} 
   1  & 1 \\
   0 & -1 
\end{pmatrix},
\end{equation*}
\begin{equation*}
\begin{pmatrix} 
   0  & 1 \\
  -1& 0
\end{pmatrix},
\begin{pmatrix} 
   0  & 1 \\
   -1 & -1 
\end{pmatrix},
\begin{pmatrix} 
   0  & -1 \\
  1& 1
\end{pmatrix}, m\in\{0,1,2,\dots\} .
\end{equation*}

\end{statement}

\begin{lemma}\label{sqrt}

The eigenvalues of second-order integer unimodular matrices are equal to the root of unity if and only if they are equal in modulus to unity.

\end{lemma}

\begin{proof}
	 Let $ A= \begin{pmatrix} 
   a  & b \\
   c & d 
\end{pmatrix} \in Gl(2,\mathbb{Z}) $. Then, $\Delta=ad-bc$ is a   
determinant of the matrix $A$ ($\Delta=\pm1$) and $f(\lambda)=\lambda^2-\lambda(a+d)+ad-bc$ is a
characteristic polynomial of the matrix $A$. Denote by  $\lambda_{1,2}$    eigenvalues of the matrix $A$.

Suppose that there exists $n\in\mathbb{N}$ such that $\lambda_i^n=1$ $(i\in\{1,2\})$. Then, $|\lambda_i|=1$.

Let us prove the statement in the opposite direction. Let $|\lambda_i|=1$ $(i\in\{1,2\})$.

In  \cite{Kronecker},  Kronecker proved that if each of roots of an integer polynomial with the leading coefficient $1$ are equal modulo to unity, then these roots are roots of unity. Since $a,b,c,d\in\mathbb{Z}$ and  $|\lambda_1||\lambda_2|=|ad-bc|=1$, we have that all eigenvalues of the matrix $ A $ are roots of unity.

\end{proof}

The next corollaries  follow from  Statement \ref{Batt} and   Lemma  \ref{sqrt}.

\begin{corollary} Each conjugacy class of non-hyperbolic algebraic automorphisms of the two-dimensional torus is given by exactly one of the following matrices:

\begin{equation*}
M_1(m)=\begin{pmatrix} 
   1 & m \\
   0 & 1 
\end{pmatrix} ,
M_2(m)=\begin{pmatrix} 
   -1 & m \\
   0 & -1 
\end{pmatrix},
M_3=\begin{pmatrix} 
   1  & 0 \\
   0 & -1 
\end{pmatrix}
,
M_4=\begin{pmatrix} 
   1  & 1 \\
   0 & -1 
\end{pmatrix},
\end{equation*}
\begin{equation*}
M_5=\begin{pmatrix} 
   0  & 1 \\
  -1& 0
\end{pmatrix},
M_6=\begin{pmatrix} 
   0  & 1 \\
   -1 & -1 
\end{pmatrix},
M_7=\begin{pmatrix} 
   0  & -1 \\
  1& 1
\end{pmatrix}, m\in\{0,1,2,\dots\} .
\end{equation*}

\end{corollary}

Recall that a  non-identity matrix $A$ is called {\it periodic}  if there exists a number $n\in\mathbb{N}$ such that $A^n=E$. The smallest of such $n$ is called the {\it  period} of the matrix $A$.

By direct calculation, we have that for $m\neq 0$  matrices $M_1(m)$ and $M_2(m)$  are not periodic. For $m=0$, the matrix $M_1(m)$ is identical and the matrix $M_2(m)$ is a matrix of period 2. Matrices $M_3$ and $M_4$ are also matrices of period 2. Periods of   matrices  $M_5$, $M_6$  and $M_7$ are equal to 4, 3 and 6 respectively.

\begin{corollary} There are 6 classes of periodic algebraic automorphisms of the two-dimensional torus, each of which is given by exactly one of the following matrices:

\begin{equation*}
M_2(0)=\begin{pmatrix} 
   -1 & 0 \\
   0 & -1 
\end{pmatrix},
M_3=\begin{pmatrix} 
   1  & 0 \\
   0 & -1 
\end{pmatrix}
,
M_4=\begin{pmatrix} 
   1  & 1 \\
   0 & -1 
\end{pmatrix},
\end{equation*}
\begin{equation*}
M_5=\begin{pmatrix} 
   0  & 1 \\
  -1& 0
\end{pmatrix},
M_6=\begin{pmatrix} 
   0  & 1 \\
   -1 & -1 
\end{pmatrix},
M_7=\begin{pmatrix} 
   0  & -1 \\
  1& 1
\end{pmatrix} .
\end{equation*}

\end{corollary}

Put $A_1=M_2(0)$, $A_2=M_6^{-1}$, $A_3=M_6$, $A_4=M_7$, $A_5=M_7^{-1}$, $A_6=M_5^{-1}$, $A_7=M_5$.

	{\bf Proof of Theorem \ref{t3}.}  
It follows from statement \ref{clNi}   that two periodic surface transformations are conjugate by an orientation preserving homeomorphism if and only if their complete characteristics coincide. It follows from Corollary \ref{ng} that any orientation preserving non-homotopic to the identity periodic homeomorphism of the two-dimensional torus satisfies exactly one of the complete characteristics $\kappa_j$  $(j=\overline{1,7})$. Let us show that the map $f_{A_j}$  has the complete characteristic $\kappa_j$  and, therefore, any orientation preserving non-homotopic to the identity periodic homeomorphism of the two-dimensional torus   conjugates by means of an orientation preserving homeomorphism with exactly one automorphism $f_{A_j}$.

Let us find the complete characteristic of the periodic transformation $ f_ {A_5} $ induced by the matrix
$A_5=\begin{pmatrix}
			1 & 1\\
			-1 & 0\\
		\end{pmatrix}$.  The complete characteristics of   automorphisms given by   remaining matrices are found in a similar way.

The period of the $A_5$  is 6. Consequently, it induces a periodic homeomorphism of the two-dimensional torus $f_{A_5}$ of period 6. 

Let us find the set $\bar B_{f_ {A_5}}$ with period less than the period of the homeomorphism $f_{A_5}$.  As mentioned before,   periods of such points  are divisors of six, that is, they are fixed or have the period  2 or 3.

We find fixed points from the system

\begin{equation*}
 \begin{cases}
   x=x+y \pmod 1
   \\
   y=-x  \pmod 1
 \end{cases},
 \end{equation*}
 
decomposing into a countable set of systems

\begin{equation*}
 \begin{cases}
   x+k=x+y, k\in\mathbb{Z}
   \\
   y+m=-x, m\in\mathbb{Z}
 \end{cases}.
 \end{equation*}

These systems are equivalent to

\begin{equation*}
 \begin{cases}
   x=-m-k
   \\
   y=k
 \end{cases} ,  k\in\mathbb{Z}, m\in\mathbb{Z} .
 \end{equation*}
 
Whence, by direct calculation we see that the map  $f_{A_5}$ has a unique fixed point $p(0,0)$, where $p:\mathbb{R}^2 \rightarrow \mathbb{T}^2$ is the natural projection.

 Finding   fixed points of the map  $f_{A_5}^2$, induced by the matrix  $A_5^2=
\begin{pmatrix} 
  0 & 1 \\
   -1 & -1 
\end{pmatrix}$, we get points of period 2 of the map  $f_{A_5}$.

By analogy with finding fixed points of the map $f_{A_5}$, we consider a countable set of systems

\begin{equation*}
 \begin{cases}
   x+k=y, k\in\mathbb{Z}
   \\
   y+m=-x-y, m\in\mathbb{Z}
 \end{cases}. 
 \end{equation*}

These systems are equivalent to

\begin{equation*}
 \begin{cases}
   x=\frac{1}{3}(-2m-k)
   \\
    y=\frac{1}{3}(k-m)
 \end{cases} ,  k\in\mathbb{Z}, m\in\mathbb{Z} .
 \end{equation*}

 Whence, by direct calculation we make sure that the map  $f_{A_5}^2$ has three fixed points $p(0,0),p(\frac{1}{3},\frac{1}{3}), p(\frac{2}{3},\frac{2}{3})$. Therefore, the map  $f_{A_5}$ has a unique orbit of period two $\mathcal O_2=\{p(\frac{1}{3},\frac{1}{3}), p(\frac{2}{3},\frac{2}{3})\}$.

 Finding fixed points of the map  $f_{A_5}^3$ induced by the matrix  $A_5^3=
\begin{pmatrix} 
   -1 & 0 \\
   0& -1 
\end{pmatrix}$, we get points of period 3 of the map  $f_{A_5}$.

By analogy with finding fixed points of the map $f_{A_5}$, we consider a countable set of systems

\begin{equation*}
 \begin{cases}
   x+k=-x, k\in\mathbb{Z}
   \\
   y+m=-y, m\in\mathbb{Z}
 \end{cases}.
 \end{equation*}

These systems are equivalent to 

 \begin{equation*}
 \begin{cases}
   x=-\frac{1}{2}k
   \\
   y=-\frac{1}{2}m
 \end{cases}   ,  k\in\mathbb{Z}, m\in\mathbb{Z} .
 \end{equation*}

Whence,  by direct calculation we have  that the map $f_{A_5}^3$ has 4 fixed points: $p(0,0),p(\frac{1}{2},\frac{1}{2}), p(\frac{1}{2},0), p(0,\frac{1}{2})$. Therefore, the map $f_{A_5}$ has a unique orbit $\mathcal O_3=\{p(\frac{1}{2},\frac{1}{2}), p(\frac{1}{2},0), p(0,\frac{1}{2})\}$ of period 3.

Thus, the set $\bar B_{f_{A_5}}$ of the map $f_{A_5}$ consists of three orbits:
 $\mathcal O_1=\{p(0,0)\}, \mathcal O_2=\{p(\frac{1}{3},\frac{1}{3}), p(\frac{2}{3},\frac{2}{3})\}, \mathcal O_3=\{p(\frac{1}{2},\frac{1}{2}), p(\frac{1}{2},0), p(0,\frac{1}{2})\}$.

Consider the map  $\hat A_1:
 \begin{cases}
   \overline{x}=x+y
   \\
  \overline{y}=-y
 \end{cases}.$
 This map is covering with respect to $f_{A_5}$.  The point $O_1(0,0)$ is fixed under $\hat A_1$.

In some neighborhood of the point $O_1$ we fix the point $Q(1,0)$.  Under the action of $\hat A_1$ the orbit of the point $O_1$ is the set 
\begin{equation*}
\mathcal{O}_{Q}=\{(1,0),(0,1),(-1,1),(-1,0),(0,-1),(1,-1)\}. \end{equation*}  Connect  each of points of the set $\mathcal{O}_{Q}$ by means of segment to the point $O_1$. Consider a circle of center $O_1$ such that it intersects each of constructed segments at some point  (see Fig. \ref{A1}). 

\begin{figure}[h!]\center{\includegraphics[width=0.5\linewidth]{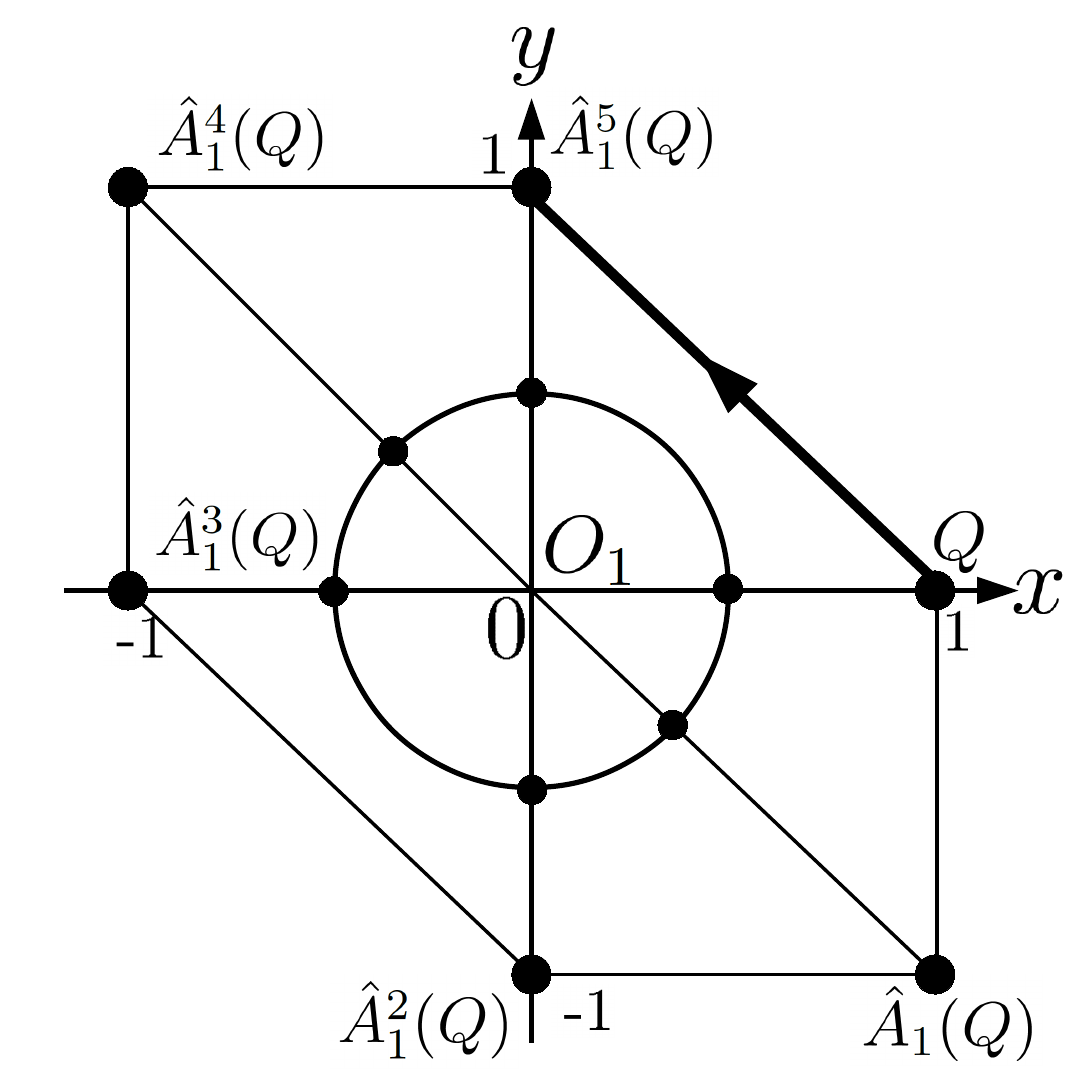}}\caption{The action of $\hat A_1$ in some neighborhood of the fixed point. }\label{A1}\end{figure}

Connect  points of the set $\mathcal{O}_{Q}$ in pairs  by means of segment following the counterclockwise order on the circle.  We obtain a closed curve (see Fig. \ref{A1}). It is an invariant curve of the map $\hat A_1$. The  open arc $(Q,\hat A_1^5(Q))$ does not contain points of the orbit  of the point  $Q$  in the counterclockwise direction .  Therefore, $d_3=5$.

Consider the map  $\hat A_2:
 \begin{cases}
   \overline{x}=y
   \\
  \overline{y}=-x-y+1
 \end{cases}.
$   This map is covering with respect to $f_{A_5}^2$.  The point $O_2(\frac{1}{3}, \frac{1}{3})$ is fixed under $\hat A_2$.

Analogously to finding $d_3$, we construct a closed curve invariant under the map $f_{A_5}^2$  and find the parameter $d_2=2$.

Consider the map $\hat A_3:
 \begin{cases}
   \overline{x}=-x+1
   \\
  \overline{y}=-y+1
 \end{cases}.
$    This map is covering with respect to $f_{A_5}^3$.  The point $O_3( \frac{1}{2}, \frac{1}{2})$ is fixed under $\hat A_3$.

The map $\hat A_3$ defines the rotation by the angle $\pi$ around the point  $O_3( \frac{1}{2}, \frac{1}{2})$  of the plane. From here we find  $d_1=1$. 

Thus, the complete characteristic of the map $f_{A_5}$ given by the matrix $A_5$ is the following  collection of the numbers: $n=6,\, p=1 , \, n_1=3, n_2=2, n_3=1,\,d_1=1,\,d_2=2,\,d_3=5$ (the complete characteristic $\kappa_5$).

\noindent\textbf{Acknowledgements}\\
\textit{The publication was prepared within the framework of the Academic Fund Program at the HSE University in 2021-2022 (grant  21-04-004).}

\bigskip

\end{document}